\documentclass[11pt]{article}

\usepackage[top=1in, left=1.20in, bottom=1in, right=1.20in]{geometry}

\usepackage{hyperref}  
\usepackage{setspace} 
\usepackage{color}
\usepackage{graphicx}
\usepackage{amsfonts}
\usepackage{amssymb}
\usepackage{amsmath}
\usepackage{amsthm}

\newtheorem{theorem}{Theorem}[section]

\newtheorem{conjecture}[theorem]{Conjecture}
\newtheorem{claim}{}[theorem]

\title{Borodin-Kostochka Conjecture holds for odd-hole-free graphs}

\author{Rong Chen, Kaiyang Lan, Xinheng Lin, Yidong Zhou\\
\\
Center for Discrete Mathematics,\ \ Fuzhou University\\
Fuzhou,\ \ P. R. China}

\begin{document}

\maketitle

\footnote{

Emails: crlwchg@163.com (R. Chen),\ \ kylan95@126.com (K. Lan),\ \ 602503578@qq.com (X. Lin),\ \ zoed98@126.com (Y. Zhou).
}

\begin{abstract}
The Borodin-Kostochka Conjecture states that for a graph $G$, if $\Delta(G)\geq 9$, then $\chi(G)\leq\max\{\Delta(G)-1,\omega(G)\}$.
In this paper, we prove the Borodin-Kostochka Conjecture holding for odd-hole-free graphs.

\end{abstract}

\section{Introduction}

All graphs in this paper are finite and simple.
For two graphs $G$ and $H$, we say that $G$ {\em contains} $H$ if $H$ is isomorphic to an induced subgraph of $G$.
When $G$ does not contain $H$, we say that $G$ is {\em $H$-free}.
For a family $\mathcal{H}$ of graphs, we say that $G$ is $\mathcal{H}$-{\em free} if $G$ is $H$-free for every graph $H\in \mathcal{H}$.

For a graph $G$, we use $\chi(G)$, $\omega(G)$ and $\Delta(G)$ to denote the chromatic number, clique number and maximum degree of $G$, respectively.
Evidently, $\omega(G)\leq \chi(G)\leq \Delta(G)+1$.
In 1941, Brooks observed that odd cycles and complete graphs are the only graphs to achieve the upper bound and strengthened this bound by proving the following result.

\vskip.1cm

\begin{theorem}[Brooks' Theorem \cite{Brook41}]\label{Brooks}
Let $G$ be a graph with $\Delta(G)\geq 3$. Then $$\chi(G)\leq\max\{\Delta(G),\omega(G)\}.$$
\end{theorem}

In 1977, Borodin and Kostochka \cite{Borodin77} conjectured that a similar result holds for $\Delta(G)-1$ colorings.

\begin{conjecture}[Borodin-Kostochka Conjecture]\label{B-K Conj}
Let $G$ be a graph with $\Delta(G)\geq 9$. Then $$\chi(G)\leq\max\{\Delta(G)-1,\omega(G)\}.$$
\end{conjecture}

Cranston, Lafayette and Rabern \cite{UG21} proved that Conjecture \ref{B-K Conj} fails under either of the weaker assumptions $\Delta(G)\geq 8$ or $\omega(G)\leq \Delta(G)-2$.
In 1999, Reed \cite{BR99} proved that Conjecture \ref{B-K Conj} holds for graphs having maximum degrees at least $10^{14}$.
Recently, the Borodin-Kostochka Conjecture was proved true for claw-free graphs in \cite{DC13}, for $\{P_5,C_4\}$-free graphs in \cite{UG21}, for $\{P_5$, gem$\}$-free graphs in \cite{DC22}, and for hammer-free graphs in \cite{CLL}.

A $cycle$ is a connected 2-regular graph.
Let $P_n$ and $C_n$ denote the path and cycle on $n$ vertices, respectively.
The {\em length} of a path or a cycle is the number of its edges.
A {\em hole} in a graph is an induced cycle of length at least four. We say a hole $C$ is {\em odd } if $|V(C)|$ is odd.
For any odd-hole-free graph $G$, we have $\chi(G)\leq 2^{2^{\omega(G)+2}}$ by the main result proved by Scott and Seymour in \cite{AS16}, while Ho\'ang \cite{AS} conjectured that $\chi(G)\leq \omega(G)^2$.
In this paper, we prove that the Borodin-Kostochka Conjecture holds for odd-hole-free graphs.

\vskip.1cm

\begin{theorem}\label{Main1}
Let $G$ be an odd-hole-free graph with $\Delta(G)\geq 9$.
Then $$\chi(G)\leq\max\{\Delta(G)-1,\omega(G)\}.$$
\end{theorem}

In fact, to prove Theorem \ref{Main1}, we prove a slightly stronger result.

\begin{theorem}\label{Main2}
Let $G$ be an odd-hole-free graph with $\Delta(G)\geq 7$.
Then $$\chi(G)\leq\max\{\Delta(G)-1,\omega(G)\}.$$
\end{theorem}

\section{Proof~of~Theorem~\ref{Main2}}
For a graph $G$ and a subset $X$ of $V(G)$, let $G-X$ denote the graph obtained from $G$ by deleting all vertices in $X$ and let $G[X]$ be the subgraph of $G$ induced by $X$.
Let $N(X)$ be the set of vertices in $V(G)-X$ that have a neighbour in $X$.
Set $N[X]:=N(X)\cup X$.
For any $x\in V(G)$, set $d_G(x):=|N(x)|$.
When there is no confusion, subscripts are omitted. For a vertex $u\in V(G)-X$, we say that $u$ is {\em complete} to $X$ if $u$ is adjacent to every vertex in $X$.
For an positive integer $k$, 
a graph $G$ is said to be {\em k-vertex-critical} if $\chi(G)=k$ and $\chi(G-v)\leq k-1$ for each vertex $v$ of $G$.


\begin{proof}[$\mathbf{Proof~of~Theorem~\ref{Main2}}$.]
When $\omega\geq \Delta(G)$, the result holds from Theorem \ref{Brooks}.
So we may assume that $\chi(G)<\Delta(G)$.
Assume that Theorem \ref{Main2} is not true.
Let $G$ be a counterexample to Theorem \ref{Main2} with $|V(G)|$ as small as possible.
Then $G$ is connected.

\begin{claim}\label{critical}
$G$ is $\Delta(G)$-vertex-critical.
\end{claim}
\begin{proof}[Subproof.]
By Theorem \ref{Brooks}, we have $\omega(G)\leq \Delta(G)-1$ and $\chi(G)=\Delta(G)$.
Let $u$ be an arbitrary vertex of $G$. Since $\chi(G-\{u\})\leq max\{\omega(G-\{u\}),\Delta(G-\{u\})\}\leq \Delta(G)-1$ by Theorem \ref{Brooks},
$\chi(G-\{u\})\leq \Delta(G)-1<\chi(G)$, implying that $G$ is $\Delta(G)$-vertex-critical.
\end{proof}

Let $u\in V(G)$ such that $d(u)=\Delta(G)$. Set $N_G(u):=\{u_1,u_2,\ldots,u_{\Delta(G)}\}$ and $G':=G-\{u\}$.
Then $G'$ has a proper $(\Delta(G)-1)$-coloring $\varphi:V(G')\rightarrow \{1,2,\ldots,\Delta(G)-1\}$ by \ref{critical}. 
If, in this coloring of $G'$, one of the $\Delta(G)-1$ colors is not assigned to a neighbour of $u$, we may assign it to $u$, thereby extending the proper $(\Delta(G)-1)$-coloring $\varphi$ of $G'$ to a proper $(\Delta(G)-1)$-coloring of $G$, which is a contradiction.
We may therefore assume that the $\Delta(G)$ neighbours of $u$ receive all $\Delta(G)-1$ colors.
Without loss of generality, let $\varphi(u_i)=i$ for each $1\leq i\leq \Delta(G)-1$ and $\varphi(u_{\Delta(G)})=\Delta(G)-1$.
Set $V_i:=\{u_i\}$ for $1\leq i\leq \Delta(G)-2$ and $V_{\Delta(G)-1}:=\{u_{\Delta(G)-1},u_{\Delta(G)}\}$. 
That is, $V_i$ is the set consisting of the vertices of $N(u)$ which are assigned color $i$. 

\begin{claim}\label{complete}
$[V_i,V_j]\neq \emptyset$ for any $1\leq i<j\leq \Delta(G)-1$, where $[V_i,V_j]$ denotes the set of edges in $G$ that has one end in $V_i$ and other end in $V_j$.
\end{claim}
\begin{proof}[Subproof.]
Suppose for a contradiction that there exist $V_i,V_j$ such that $[V_i,V_j]=\emptyset$.
Denote by $G_{ij}$ the subgraph of $G'$ induced by all vertices assigned colors $i$ or $j$.
Let $C$ be the component of $G_{ij}$ that contains $u_i$.
Then $V(C)\cap V_j\neq \emptyset$.
If not, by interchanging the colors $i$ and $j$ in $C$, we obtain a new $(\Delta(G)-1)$-coloring of $G'$ in which only $\Delta(G)-2$ colors (all but $i$) are assigned to the neighbours of $u$, which is a contradiction.
Therefore, at least one vertex of $V_j$ is contained in $C$.
Let $P_{ij}$ be a shortest induced path in $C$ linking $u_i$ and a vertex in $V_j$.
Clearly $P$ has odd length.
Since $[V_i,V_j]=\emptyset$, we have that $P$ has length at least 3, so $G[V(P)\cup\{u\}]$ is an odd hole, which is a contradiction.
\end{proof}

\ref{complete} implies that $G[\{u,u_1,u_2,\ldots,u_{\Delta(G)-2}\}$ is a $(\Delta(G)-1)$-clique.

Let $\varphi'$ be another proper $(\Delta(G)-1)$-coloring of $G'$. By the symmetry between $\varphi$ and $\varphi'$, there exist exactly two vertices $x,y\in N_{G}(u)$ with $\varphi'(x)=\varphi'(y)$.
\begin{claim}\label{s3+}
$V_{\Delta(G)-1}=\{x,y\}$. 
\end{claim}
\begin{proof}[Subproof.]
Assume not.
Since $xy\notin E(G)$, by \ref{complete}, we have $|V_{\Delta(G)-1}\cap \{x,y\}|=1$.
Without loss of generality, we may assume that $x=u_{\Delta(G)}$.
By \ref{complete} and symmetry again, $u_{\Delta(G)-1}$ is complete to $N_G[u]-\{y, u_{\Delta(G)-1},u_{\Delta(G)}\}$.
Since $y$ is not adjacent to $u_{\Delta(G)}$, it follows from \ref{complete} that $yu_{\Delta(G)-1}\in E(G)$.
Hence, $N_G[u]-\{u_{\Delta(G)}\}$ induces a clique of size $\Delta(G)$, which is a contradiction.
\end{proof}

\begin{claim}\label{at most one}
For any $1\leq i\leq \Delta(G)$, there are at most a pair of vertices in $N_{G'}(u_i)$  that can be assigned the same color.
\end{claim}
\begin{proof}[Subproof.]
When $1\leq i\leq \Delta(G)-2$, since  $\Delta(G)-2\leq d_{G'}(u_i)\leq \Delta(G)-1$ and $u_i$ has at most one non-neighbour in $N_G[u]-\{u_i\}$ by \ref{complete}, $u_i$ has at most one neighbour in $V(G)-N_G[u]$.
Moreover, if $u_i$ is complete to $V_{\Delta(G)-1}$, then $N_G[u_i]=N_G[u]$.
So \ref{at most one} holds when $1\leq i\leq \Delta(G)-2$.
Hence, we may assume that $\Delta(G)-1\leq i\leq \Delta(G)\}$. 
By symmetry it suffices to show that \ref{at most one} holds when $i=\Delta(G)$.
Suppose not. 
Then there must exists a color, say $j$, assigned to no vertex in $N_{G'}[u_{\Delta(G)}]$ as $\Delta(G)-2\leq d_{G'}(u_{\Delta(G)})\leq \Delta(G)-1$.
Hence, we can recolor $u_{\Delta(G)}$ by $j$ to obtain a new proper $(\Delta(G)-1)$-coloring of $G'$, which is a contradiction to \ref{s3+}.
\end{proof}



By\ref{complete} and the Pigeonhole Principle, without loss of generality we may assume that $\{u_1,u_2,u_3\}\subseteq N_G(u_{\Delta(G)})$ as $\Delta(G)\geq 7$.
Moreover, since $G$ has no $\Delta(G)$-clique, there exists some $4\leq i\leq\Delta(G)-2$ such that $u_i u_{\Delta(G)}\notin E(G)$. 
So $u_i$ is complete to $N_G[u]-\{u_{\Delta(G)}\}$ by \ref{complete}. By \ref{at most one} and symmetry we may assume that $u_1$ is the unique vertex in $N_{G'}(u_i)\cup N_{G'}(u_{\Delta(G)})$ assigned color $1$. By \ref{at most one} again, either $\Delta(G)-1$ or $i$, say $i$, is used exactly once in  $N_{G'}(u_1)$. Hence, we can recolor $u_i,u_{\Delta(G)}$ by $1$ and $u_1$ by $i$ to obtain a proper coloring of $G'$, which is a contradiction to \ref{s3+}.
This completes the proof of Theorem \ref{Main2}.
\end{proof}

\section{Acknowledgments}
This research was partially supported by grants from the National Natural Sciences Foundation of China (No. 11971111). 


\end{document}